%% file: doubly-chordal-graphs_v3.tex
\documentclass[review]{elsarticle}

\input{macros}

\usepackage{tikz}
\usetikzlibrary{positioning,chains,fit,shapes,calc}
\tikzstyle{vertex}=[circle, draw, inner sep=0pt, minimum size=6pt]

\usepackage{algpseudocode}
\usepackage{algorithmicx}

\usepackage{algorithm}

\usepackage[normalem]{ulem}

\newtheorem*{theorem*}{Theorem}

\tikzstyle{vertex}=[circle, draw, inner sep=0pt, minimum size=6pt]

\usepackage{lineno,hyperref}
\modulolinenumbers[5]

\journal{arxiv}

\begin{document}

\begin{frontmatter}

\title{Efficient Algorithm for Checking 2-Chordal (Doubly Chordal) Bipartite Graphs}

%% Authors per affiliation:
\author{Austin Alderete}
\corref{mycorrespondingauthor}
\cortext[mycorrespondingauthor]{Corresponding author}
\ead{aaldere@math.utexas.edu}
\address{University of Texas at Austin}

\begin{abstract}
\small{
We present an algorithm for determining whether a bipartite graph $G$ is 2-chordal (formerly doubly chordal bipartite). At its core this algorithm is an extension of the existing efficient algorithm for determining whether a graph is chordal bipartite. We then introduce the notion of $k$-chordal bipartite graphs and show by inductive means that a slight modification of our algorithm is sufficient to detect this property. We show that there are no nontrivial $k$-chordal bipartite graphs for $k \geq 4$ and that both the $2$-chordal bipartite and $3$-chordal bipartite problem are contained within complexity class $\mathrm{P}$.}
\end{abstract}

\begin{keyword}
doubly chordal \sep 2-chordal
\end{keyword}

\end{frontmatter}

\section{Introduction}

The study of chordal graphs can be traced back to 1958 when they were first introduced by Hajnal and Sur\'{a}nyi. These are the graphs for which every cycle of length four or greater contains a chord \cite{Hajnal}. Since then several subclasses and variations of chordal graphs, notably bipartite analogs, have been introduced. Of interest to us is the notion of a 2-chordal bipartite (also known as a doubly chordal bipartite) graph which was recently connected to discrete statistical models and their maximum likelyhood degrees \cite{Coons}. Efficient algorithm for recognizing chordal graphs and chordal bipartite graphs are known and are of polynomial complexity on the size of the vertex set\cite{Farber} \cite{Dirac} \cite{Pelsmajer}. 2-chordal bipartite graphs can be characterized as those bipartite graphs which are chordal bipartite and lack a double-square graphs as an induced subgraph, however the induced subgraph problem is known to be NP-complete in general \cite{Coons} \cite{David}.

In this work, we present an efficient algorithm for checking whether a graph is 2-chordal bipartite and as a corollary show that the induced subgraph problem for the double-square graph is solvable in polynomial time. The algorithm breaks into two components- the first being the check that the graph is chordal bipartite and the second being an iterative process which runs over all edge deletions and is based on the vertex characterization due to Farber. We also show that the conditions checked are both necessary and sufficient conditions (although there may be a faster check for them) and then characterize the runtime of the algorithm. The key theorem underlying this algorithm is Theorem \ref{thm:mainThm}.

\newtheorem*{thm:mainThm}{Theorem \ref{thm:mainThm}}
\begin{thm:mainThm}[2-Chordal Bipartite Identification Theorem]
Given a bipartite graph $G = (X,Y,E)$, it is 2-chordal bipartite if and only if both of the following are true:
\begin{itemize}
\item[(i)] $G$ is chordal bipartite.
\item[(ii)] for each $e \in E$, the graph $G_e = (X,Y,E-\{e\})$ obtained from $G$ by deleting edge $e$ is chordal bipartite.
\end{itemize}
\end{thm:mainThm}

We also introduce and explore the natural generalization of Theorem \ref{thm:mainThm} to $k$-chordal bipartite graphs. This generalization allows us to propose algorithms which check the property of being $k$-chordal bipartite and we show that they are also in complexity class $P$ for all $k$. However, we provide a proof that there are no nontrivial $k$-chordal bipartite graphs for $k \geq 4$ where nontrivial means that they contain a cycle of size at least $6$ and in doing so remove the need for anything other than the $k=2,3$ algorithms.

\section{Background}

In this section we lay out a few foundational results about chordal and strongly chordal graphs as well as their connection to chordal bipartite graphs. Formally, a \textit{chordal graph} is a simple graph possessing no chordless cycles of length four or greater. Such a cycle is called a hole and so a chordal graph is a graph without holes \cite{Hajnal}. A subclass of these graphs, introduced by Farber, is the \textit{strongly chordal} graphs; these are chordal graphs for which every even cycle of length at least 6 has a strong chord, i.e. a chord joining vertices whose distance along the cycle is odd \cite{Farber}. As we will see, strongly chordal graphs are closely related to another subclass found among bipartite graphs.

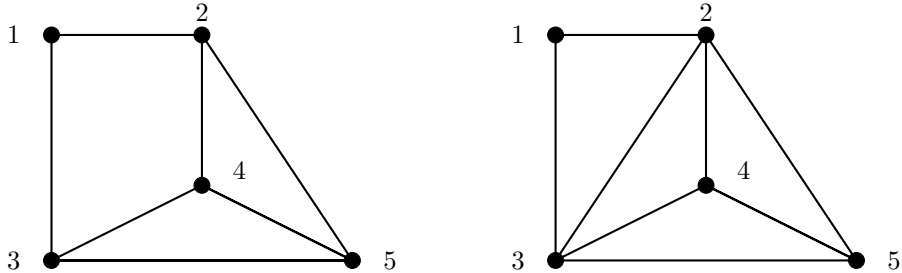
\begin{figure}[h]
\centering
\begin{tikzpicture}
%% vertices
\draw[fill=black] (0,0) circle (3pt);
\draw[fill=black] (4,0) circle (3pt);
\draw[fill=black] (2,1) circle (3pt);
\draw[fill=black] (2,3) circle (3pt);
\draw[fill=black] (0,3) circle (3pt);
%% vertex labels
\node at (-0.5,3) {1};
\node at (-0.5,0) {3};
\node at (4.5,0) {5};
\node at (2.5,1.2) {4};
\node at (2,3.3) {2};
%%% edges
\draw[thick] (0,0) -- (4,0) -- (2,1) -- (0,0)  -- (4,0) -- (2,1) -- (2,3);
\draw[thick] (0,0) -- (0,3) -- (2,3) -- (4,0);
\end{tikzpicture}
~~~~~~~~~
\begin{tikzpicture}
%% vertices
\draw[fill=black] (0,0) circle (3pt);
\draw[fill=black] (4,0) circle (3pt);
\draw[fill=black] (2,1) circle (3pt);
\draw[fill=black] (2,3) circle (3pt);
\draw[fill=black] (0,3) circle (3pt);
%% vertex labels
\node at (-0.5,3) {1};
\node at (-0.5,0) {3};
\node at (4.5,0) {5};
\node at (2.5,1.2) {4};
\node at (2,3.3) {2};
%%% edges
\draw[thick] (0,0) -- (4,0) -- (2,1) -- (0,0) -- (2,3) -- (4,0) -- (2,1) -- (2,3);
\draw[thick] (0,0) -- (0,3) -- (2,3);
\end{tikzpicture}
\caption{Non-chordal (left) versus chordal (right) graph}
\end{figure}

We will put a name to two special types of vertices: simplicial vertices and simple vertices. A \textit{simplicial vertex} is a vertex whose open neighborhood forms a clique; that is, a vertex $v$ in the graph $G$ is simplicial if the subgraph induced on $N_G(v)$ is a complete graph. Meanwhile, a \textit{simple vertex} is a vertex whose neighbors have closed neighborhoods which can be ordered as a chain under inclusion; that is, a vertex $v$ in the graph $G$ with open neighborhood $N_G(v)= \{ v_1,...,v_k \}$ is simple if there exists some relabeling $\sigma$ on $[k]$ such that
\[
N_G[v_{\sigma(1)}]
\subseteq
N_G[v_{\sigma(2)}]
\subseteq
...
\subseteq
N_G[v_{\sigma(k)}]
\]
Observe that every simple vertex is also simplicial. The proof is as follows, suppose $v$ is a simple vertex and consider the induced subgraph on $N_G(v)$. It suffices to show that for any pair of edges $u,w \in N_G(v)$ there is an edge between them. As the set of closed neighbhorhoods of neighbors of $v$ form a chain, we have either $u \in N_G[u] \subseteq N_G[w]$ or $w \in N_G[w] \subseteq N_G[u]$. In either case, $uw \in E(G)$ as desired.

\begin{figure}[h]
\centering
\begin{tikzpicture}
%% vertices
\draw[fill=black] (0,0) circle (3pt);
\draw[fill=black] (4,0) circle (3pt);
\draw[fill=black] (2,1) circle (3pt);
\draw[fill=black] (2,3) circle (3pt);
\draw[fill=black] (7,0) circle (3pt);
%% vertex labels
\node at (0,-0.5) {1};
\node at (4,-0.5) {3};
\node at (2.5,1.2) {4};
\node at (2,3.3) {2};
\node at (7,-0.5) {5};
%%% edges
\draw[thick] (0,0) -- (4,0) -- (2,1) -- (0,0)  -- (4,0) -- (2,1) -- (2,3);
\draw[thick] (0,0) -- (2,3) -- (4,0);
\draw[thick] (4,0) -- (7,0);
\end{tikzpicture}
\caption{Graph with simple and simplicial vertex 4}
\end{figure}
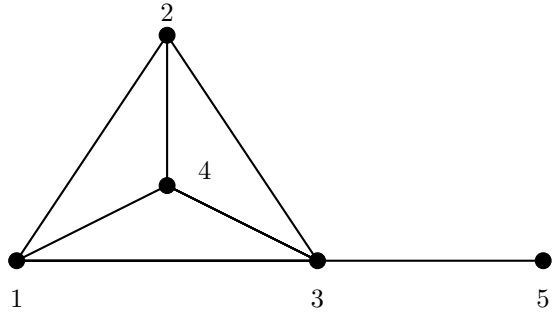

\begin{figure}[h]
\centering
\begin{tikzpicture}
%% vertices
\draw[fill=black] (0/2,0/2) circle (3pt);
\draw[fill=black] (4/2,0/2) circle (3pt);
\draw[fill=black] (2/2,1/2) circle (3pt);
\draw[fill=black] (2/2,3/2) circle (3pt);
\draw (7/2,0/2) circle (3pt);
%% vertex labels
\node at (0,-0.3) {1};
\node at (2,-0.3) {3};
\node at (1.25,0.6) {4};
\node at (1,1.8) {2};
\node at (3.5,-0.3) {5};
%%% edges
\draw[thick] (1,0.5) -- (0,0) -- (1,1.5) -- (2,0) -- (1,0.5);
\draw[thick] (1,1.5) -- (1,0.5);
\draw[thick] (0,0) -- (2,0);
\end{tikzpicture}
~~~~ $\mbox{\Huge $ \subseteq$ }$~~~~
\begin{tikzpicture}
%% vertices
\draw[fill=black] (0/2,0/2) circle (3pt);
\draw[fill=black] (4/2,0/2) circle (3pt);
\draw[fill=black] (2/2,1/2) circle (3pt);
\draw[fill=black] (2/2,3/2) circle (3pt);
\draw (7/2,0/2) circle (3pt);
%% vertex labels
\node at (0,-0.3) {1};
\node at (2,-0.3) {3};
\node at (1.25,0.6) {4};
\node at (1,1.8) {2};
\node at (3.5,-0.3) {5};
%%% edges
\draw[thick] (1,0.5) -- (0,0) -- (1,1.5) -- (2,0) -- (1,0.5);
\draw[thick] (1,1.5) -- (1,0.5);
\draw[thick] (0,0) -- (2,0);
\end{tikzpicture}
~~~~ $\mbox{\Huge $ \subseteq$ }$ ~~~~
\begin{tikzpicture}
%% vertices
\draw[fill=black] (0/2,0/2) circle (3pt);
\draw[fill=black] (4/2,0/2) circle (3pt);
\draw[fill=black] (2/2,1/2) circle (3pt);
\draw[fill=black] (2/2,3/2) circle (3pt);
\draw[fill=black] (7/2,0/2) circle (3pt);
%% vertex labels
\node at (0,-0.3) {1};
\node at (2,-0.3) {3};
\node at (1.25,0.6) {4};
\node at (1,1.8) {2};
\node at (3.5,-0.3) {5};
%%% edges
\draw[thick] (1,0.5) -- (0,0) -- (1,1.5) -- (2,0) -- (1,0.5);
\draw[thick] (1,1.5) -- (1,0.5);
\draw[thick] (0,0) -- (2,0);
\draw[thick] (2,0) -- (3.5,0);
\end{tikzpicture}
\caption{$N_G[1] \subseteq N_G[2] \subseteq N_G[3]$ for neighbors of vertex 4 of Figure 2}
\end{figure}
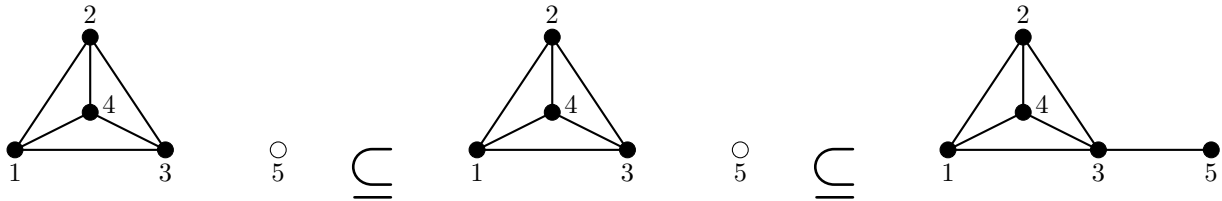

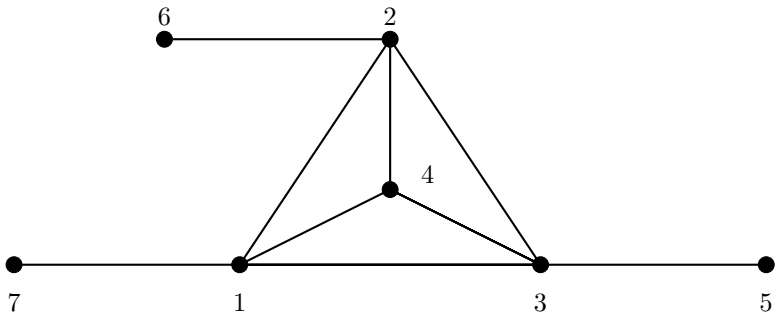
\begin{figure}[h]
\centering
\begin{tikzpicture}
%% vertices
\draw[fill=black] (0,0) circle (3pt);
\draw[fill=black] (4,0) circle (3pt);
\draw[fill=black] (2,1) circle (3pt);
\draw[fill=black] (2,3) circle (3pt);
\draw[fill=black] (7,0) circle (3pt);
\draw[fill=black] (-1,3) circle (3pt);
\draw[fill=black] (-3,0) circle (3pt);
%% vertex labels
\node at (0,-0.5) {1};
\node at (4,-0.5) {3};
\node at (2.5,1.2) {4};
\node at (2,3.3) {2};
\node at (7,-0.5) {5};
\node at (-1,3.3) {6};
\node at (-3,-0.5) {7};
%%% edges
\draw[thick] (0,0) -- (4,0) -- (2,1) -- (0,0)  -- (4,0) -- (2,1) -- (2,3);
\draw[thick] (0,0) -- (2,3) -- (4,0);
\draw[thick] (4,0) -- (7,0);
\draw[thick] (0,0) -- (-3,0);
\draw[thick] (2,3) -- (-1,3);
\end{tikzpicture}
\caption{Graph whose vertex 4 is simplicial but not simple}
\end{figure}

These distinguished vertex classes play an important role in the theory of chordal graphs. A graph is chordal if and only if every induced subgraph contains a simplicial vertex and a graph is strongly chordal if and only if every induced subgraph has a simple vertex \cite{Farber} \cite{Dirac}. These properties lend themselves to finding chordal and strongly chordal graphs by means of algorithms which perform a simplicial elimination ordering or a simple elimination ordering respectively; that is, an algorithm to determine whether a graph is chordal (strongly chordal) is to search the graph for a simplicial (simple) vertex, form the subgraph induced by removing it, and repeat. Such an algorithm effectively checks either that there are sufficiently many simplicial vertices- i.e. every induced subgraph contains at least one of them- or it identifies a particular induced subgraph which does not contain a simplicial vertex. Note that the algorithm does not run over all induced subgraphs nor is it a requirement to do so.

These algorithms provide polynomial-time methods of determining characteristics of graphs. For strong chordal graphs, the simple elimination ordering is equivalent to the \textit{strong elimination ordering} which is an ordering $v_1,...,v_n$ of the vertices such that if (i) $i<j$ and $k<l$, (ii) $v_k,v_l \in N[v_i]$, (iii) $v_j \in N[v_k]$, then $v_l \in N[v_j]$ \cite{Pelsmajer}. There are numerous other characterizations of strongly chordal graphs which include sun-characterizations, hereditary dually chordal characterizations, hypertree characterizations, and $(n+4)$-vertex cycle subgraph characterizations \cite{isgci}. One characterization that we will briefly touch on is the hypergraph characterization. A graph is chordal if and only if the hypergraph of its maximal cliques is the dual of a hypertree. This leads us to the definition of a dually chordal graph, which is a graph whose hypergraph of maximal cliques is itself a hyperrtree. A graph is called \textit{doubly chordal} if it is both chordal and dually chordal. This naming convention will influence our choice of nomenclature when it comes to bipartite graphs.

We now turn our attention to bipartite graphs proper and their subclasses. A bipartite graph is \textit{chordal bipartite} if every cycle of length greater than or equal to 6 has a chord. A bipartite graph is \textit{2-chordal bipartite} if every cycle of length greater than or equal to 6 has at least two chords. This property has also been called ``doubly chordal bipartite'' in the literature, but due to the classical definition of doubly chordal mentioned above we have adopted ``2-chordal bipartite'' in its place \cite{Coons}. If $G = (X,Y,E)$ is a bipartite graph and $G'$ is the graph obtained from $G$ by adding an edge in-between every pair of vertices in $X$, then $G$ is chordal bipartite if $G'$ is strongly chordal \cite{isgci}. This provides a link between our concepts. An important observation is that a chordal bipartite graph need not be chordal. For example, consider the graph $C_4$ with $V(C_4) := \{a,b,c,d\}$ and $ E(C_4):= \{ \{a,b\},\{c,d\},\{ac,ad,bc,bd \} \}$. Clearly this graph is vacuously chordal bipartite as it contains no cycles of length 6. However, it contains the cycle $acbd= ac+cb+bd+ac$, of length 4, and this cycle contains no chords (in fact, it uses every edge of $C_4$ as highlighted by the addition).

\section{Algorithm for determining whether a graph is 2-chordal bipartite}

We begin with our main algorithm which takes a bipartite graph as input and answers the yes/no question of whether or not the graph is 2-chordal bipartite. We define two classes beforehand. The first is the class \textit{graph} which is initialized with a triple containing its name, vertex set, and edge set given as pairs of vertices. For its methods it has the usual graph operations as well as a method called \textit{simple} which takes as input a vertex and outputs a boolean value representing whether or not the vertex is simple. The second class is \textit{biGraph} which is initialized with a 4-tuple containing its name, two disjoint vertex sets, and its edge set given as pairs selected from both sets. Its methods are the usual operations on bipartite graphs as well as a forgetful operation which constructs its underlying graph. With this in mind we are ready to present the algorithm.

%\floatname{algorithm}{Algorithm}
\begin{algorithm}
\caption{2-chordal Bipartite Checker}\label{euclid}
\begin{algorithmic}[1]
\State \textbf{Input:} $G := (X,Y,E)$ a bipartite graph with vertex sets $X,Y$ and edge set $E$
\State \textbf{Output:} \textbf{True} if $G$ is 2-chordal bipartite, \textbf{false} otherwise
\State \qquad $H := \mathrm{graph}(X \cup Y,E)$
\State \qquad yetToGlue := $X$
\State \qquad \textbf{for} $x$ in $X$:
\State \qquad \qquad \qquad yetToGlue :=  yetToGlue$-\{x\}$
\State \qquad \qquad \textbf{for} $v$ in yetToGlue:
\State \qquad \qquad \qquad \textbf{add} $\{x,v\}$ to $H$.edges
\State \qquad flag := 0
\State \qquad \textbf{while} $H$.vertices $\neq \emptyset$ \textbf{and} flag $== 0$:
\State \qquad \qquad \textbf{for} vertex in $H$.vertices:
\State \qquad \qquad \qquad \textbf{if} $H$.simple(vertex) == True:
\State \qquad \qquad \qquad \qquad $H$ := $H$.delete(vertex)
\State \qquad \qquad \qquad \qquad \textbf{if} $H$.vertex == $\emptyset$:
\State \qquad \qquad \qquad \qquad \qquad flag :=1
\State \qquad \qquad \qquad \qquad \textbf{restart} while loop
\State \qquad \qquad flag := 2
\State \qquad \textbf{if} flag == 2:
\State \qquad \qquad \textbf{return} false
\State \qquad \textbf{for} edge in $E$:
\State \qquad \qquad $H_e$ := graph($X \cup Y$,$E-\{e\}$)
\State \qquad \qquad yetToGlue := $X$
\State \qquad \qquad \textbf{for} $x$ in $X$:
\State \qquad \qquad \qquad yetToGlue := yetToGlue - $\{x\}$
\State \qquad \qquad \qquad \textbf{for} $v$ in yetToGlue:
\State \qquad \qquad \qquad \qquad \textbf{add} $\{x,v\}$ to $H_e$.edges
\State \qquad \qquad flag := 0
\State \qquad \qquad \textbf{while} $H_e$.vertices $\neq \emptyset$ \textbf{and} flag == 0:
\State \qquad \qquad \qquad \textbf{for} vertex in $H_e$.vertices:
\State \qquad \qquad \qquad \qquad \textbf{if} $H_e$.simple(vertex) == True:
\State \qquad \qquad \qquad \qquad \qquad $H_e := H_e$.delete(vertex)
\State \qquad \qquad \qquad \qquad \qquad if $H_e$.vertex == $\emptyset$:
\State \qquad \qquad \qquad \qquad \qquad \qquad flag :=1
\State \qquad \qquad \qquad \qquad \qquad \textbf{restart} while loop
\State \qquad \qquad \qquad flag :=2
\State \qquad \textbf{if} flag == 2:
\State \qquad \qquad \textbf{return} false
\State \qquad \textbf{return} True

%\State \qquad \textbf{for} $C$ in $\mathcal{C}$:
%\State \qquad \qquad chainMatrix := new $(r+1)\times n$ zero matrix
%\State \qquad \qquad \textbf{for} $0 \leq$ rowIndex $\leq r$:
%\State \qquad \qquad \qquad flat := $C[$rowIndex$]$
%\State \qquad \qquad \qquad \textbf{for} element in flat:
%\State \qquad \qquad \qquad \qquad chainMatrix(rowIndex,element) :=1
%\State \qquad \qquad append chainMatrix to listOfChainMatrices
%\State \textbf{Output:} listOfChainMatrices
\end{algorithmic}
\end{algorithm}

A fast implementation of the simple method for vertices can be done in any software which can check whether a collection of sets is nested. Having introduced the algorithm we will now show that it successfully detects whether or not the graph is 2-chordal bipartite. We claim that the graph successfully checks (i) and (ii) of Theorem \ref{thm:mainThm}.

\begin{thm}[2-Chordal Bipartite Identification Theorem]
\label{thm:mainThm}
Given a bipartite graph $G = (X,Y,E)$, it is 2-chordal bipartite if and only if both of the following are true:
\begin{itemize}
\item[(i)] $G$ is chordal bipartite.
\item[(ii)] for each $e \in E$, the graph $G_e = (X,Y,E-\{e\})$ is chordal bipartite.
\end{itemize}
\end{thm}

We first look at the intuition behind this theorem. Property (ii) is not sufficient alone as we could have a graph which was entirely a cycle of size 6. In this hypothetical graph $G$, the modified graph $G_e$ would satisfy (ii) for any $e \in E(G)$ as it contains no cycles. However, the original graph $G$ clearly fails to be chordal bipartite let alone 2-chordal bipartite. With the addition of (i), we guarantee that any cycle of size 6 or more contains at least one chord. As we iterate over the edges of the graph in (ii), we encounter this chord and our modified graph contains the same cycle, but with one fewer chords. Therefore if the modified graph is chordal bipartite, it must contain a second chord.

\begin{proof}
Let $G = (X,Y,E)$ be a graph. Suppose that (i) and (ii) hold for $G$. We now wish to show that $G$ is 2-chordal bipartite. Suppose that $G$ contains a cycle $C$ of length greater than or equal to $6$. As (i) is true, it must have at least one chord, which we will denote $c$. Now consider $G_e$. Clearly $C \subseteq E(G_e)$ and so $G_e$ also contains a cycle of size 6 or greater. Moreover, by (ii) we have that this cycle contains a chord in $G_e$. As $e$ is not an edge in $G_e$, this chord must be a second chord of the cycle $C$ when viewed in $G$ itself and so we are done.

Conversely, if a graph $G$ is 2-chordal bipartite, then it automatically must satisfy (i). Supposing that (ii) is false for $G$ implies that there exists some edge $e$ for which $G_e$ fails to be chordal bipartite. In which case $G_e$ contains a cycle of length greater than or equal to 6, call it $C \subseteq E(G_e)$, which contains no chords. Consider the embedded image of $C$ in $G$. As a cycle in $G$ it contains at most one more chord than it did in $G_e$ (namely $e$), hence it contains at most a single chord, contradicting the assumption that $G$ is 2-chordal bipartite.
\end{proof}

\begin{thm}
The 2-chordal Bipartite Checker Algorithm (Algorithm 1) returns True if and only if its input is a 2-chordal bipartite graph.
\end{thm}

\begin{proof}
We first show that that Algorithm 1 halts on any bipartite graph $G$ input. This amounts to showing that each while loop ends. We look first at the while loop on line 10. If the for loop on line 11 iterates through every vertex of $H$ (which is a finite set) then the variable flag is set to 2 and the while loop terminates followed by the algorithm returning false. If the for loop on line 11 fails to iterate through every vertex of $H$, then line 16 must have been carried out. This requires that there is a vertex which satisfies the condition on line 12 (i.e. it's simple) and so line 13 runs. This reduces the cardinality of the vertex set of $H$ and then resets the while loop. Therefore the while loop on line 10 must eventually break as there are only finitely many vertices to remove. There is only one other loop which could potentially prevent the algorithm from halting and it occur on line 23. However, this loop is a copy of the loop on line 10 and so it also ends.

Suppose that Algorithm 1 is run with input a 2-chordal bipartite graph $G = (X,Y,E)$. Note that if the algorithm terminates and the variable ``flag'' is never set equal to 2, then it returns True. Thus it is only necessary to show that flag is never set equal to 2. On line 17 we have that flag is set equal to 2. For this line of the algorithm to be run requires that the for loop on line 11 complete without resetting the while loop on line 10, i.e. the conditions of line 12 must never be satisfied for any vertex in $H$. That is, $H$ contains no simple vertices. However, if $H$ contains no simple vertices then $H$ is not strongly chordal. In the algorithm, $H$ is the graph obtained from $G$ by adding edges between every pair of vertices in $X$. Therefore, if $H$ fails to be strongly chordal, $G$ fails to be chordal bipartite. By Theorem \ref{thm:mainThm}, this implies that $G$ is not 2-chordal bipartite. This contradicts our assumption and so flag must not be set equal to 2 on line 17. The only other place flag can be set equal to 2 is on line 30. This requires that there exist an edge $e$ of our graph $G$ for which the graph obtained from $G$ by deleting $G$ fails to be chordal bipartite (the algorithm performs the same strongly chordal check). But this would imply that $G$ fails to satisfy condition (ii) of Theorem \ref{thm:mainThm} and so cannot happen. Therefore, if Algorithm 1 is run with input a 2-chordal bipartite graph, then it outputs True.

Conversely if the Algorithm outputs True, then by running the above argument in reverse we find that conditions (i) and (ii) of the theorem are satisfied and so its input must have been a 2-chordal bipartite graph.
\end{proof}

\section{k-chordal bipartite graphs}

Theorem \ref{thm:mainThm} has a natural generalization to $k$-chordal bipartite graphs.

\begin{thm}
\label{kchordThm}
Given a bipartite graph $G = (X,Y,E)$, it is $k$-chordal bipartite if and only if both of the following are true:
\begin{itemize}
\item[(i)] $G$ is chordal bipartite
\item[(ii)] for each $e \in E$, the graph $G_e = (X,Y,E-\{e\})$ is $(n-1)$-chordal bipartite.
\end{itemize}
\end{thm}

We allow 0-chordal bipartite to be a vacuous truth of all graphs so that Theorem \ref{kchordThm}. is valid for all $k \in \N$. We call an $k$-chordal bipartite graph \textit{trivial} if it contains no cycles of length 6 or greater. The proof of this theorem mirrors the 2-chordal bipartite case.

\begin{proof}
We prove this by induction. We've already shown the base case, so let $k>2$ and suppose that the theorem holds for all natural numbers less than $k$. Let $G$ be a bipartite graph and suppose that (i) and (ii) hold for the $k$-chordal bipartite version of the theorem, i.e. $G$ is chordal bipartite and for each $e \in E$ the graph $G_e$ is $(k-1)$-chordal bipartite.

If $G$ has a cycle of size $6$, then by (i) it contains a chord, call it $e$. By (ii), $G_e$ contains $(k-1)$ chords and so the cycle in $G$ must contain all of those chords plus $e$, i.e. at least $k$ chords. Therefore $G$ is $k$-chordal bipartite.

Conversely, if $G$ is $k$-chordal bipartite then it is chordal bipartite. Supposing that (ii) is false for $G$ implies that there exists an edge $e$ for which $G_e$ fails to be $(k-1)$-chordal bipartite. That is, there is a cycle $C \subseteq E(G_e)$ which contains at most $k-2$ chords. Considering the image of $C$ in $G$ we find that it contains at most $k-1$ chords, a contradiction.
\end{proof}

There is an important bound on the minimum cycle size for nontrivial $k$-chordal bipartite graphs for large $n$.

\begin{lem}
There does not exist a nontrivial $n$-chordal bipartite graph containing a cycle of size exactly $6$ for $n \geq 4$.
\end{lem}

This can easily be seen by noting that a cycle of size 6 is constructed from two sets of three vertices and can at most admit three chords, as shown in Figure 5.

\begin{figure}[h]
\centering
\begin{tikzpicture}
%% vertices

\draw[fill=black] (-3,2) circle (3pt);
\draw	(0,2) circle (3pt);
\draw[fill=black] (3,2) circle (3pt);
\draw (3,0) circle (3pt);
\draw (-3,0) circle (3pt);
\draw[fill=black] (0,0) circle (3pt);
%%% edges
\draw[thick] (-3,2) -- (-3,0) -- (0,0) -- (3,0) -- (3,2) -- (0,2) -- (-3,2);
\draw[dash dot] (-3,2) -- (3,0);
\draw[dash dot] (3,2) -- (-3,0);
\draw[dash dot] (0,0) -- (0,2);
\end{tikzpicture}
\caption{Cycle of size 6 in a bipartite graph admits at most 3 chords}
\end{figure}
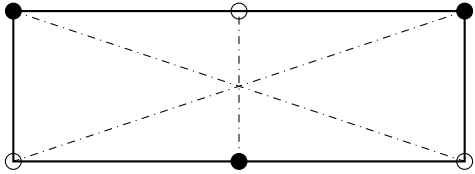

\begin{thm}
For $k \geq 4$, there are no nontrivial $k$-chordal bipartite graphs.
\end{thm}

\begin{proof}
Suppose that $G$ is a nontrivial $n$-chordal bipartite graph for fixed $n \geq 4$. Then it contains some cycle $C \subseteq E(G)$ of length greater than or equal to $6$. If the length equals $6$ then we've arrived at a contradiction by Lemma 2.3. If the length is greater than $6$ then we can choose any one of the chords which subdivide the cycle and obtain two new cycles. If one of these cycles has a length greater than or equal to 6 (but is also necessarily of length less than the length of cycle $C$), then we restart the process replacing $C$ with this new smaller cycle. 

If neither of the two cycles (which we will denote $C_1,C_2$) produced by the chord bisecting $C$ have length greater than or equal to $6$, then we consider the following: (1) $C$ is a cycle of even length, (2) the length of $C$ is greater than $6$ by the above, and (3) $|C| = |C_1|+|C_2| -2$. Therefore,
\[
8 \leq |C| = |C_1| + |C_2| - 2
\]
which implies that $10 \leq |C_1|+|C_2|$. The only way to partition $10$ into two natural numbers $|C_1|,|C_2|$ such that neither number is greater than or equal to $6$ is for $|C_1| = |C_2| = 5$. Thus, $|C| = 8$. We can draw the corresponding cycle and chord as in Figure 6.

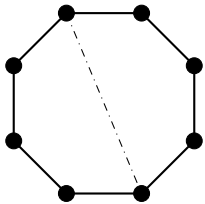
\begin{figure}[h]
\centering
\begin{tikzpicture}
%% vertices

\draw[fill=black] (1/2,2.4/2) circle (3pt);
\draw[fill=black] (-1/2,2.4/2) circle (3pt);
\draw[fill=black] (1/2,-2.4/2) circle (3pt);
\draw[fill=black] (-1/2,-2.4/2) circle (3pt);
\draw[fill=black] (2.4/2,1/2) circle (3pt);
\draw[fill=black] (-2.4/2,1/2) circle (3pt);
\draw[fill=black] (2.4/2,-1/2) circle (3pt);
\draw[fill=black] (-2.4/2,-1/2) circle (3pt);
%%% edges
\draw[thick] (-1/2,2.4/2) -- (1/2,2.4/2) -- (2.4/2,1/2) -- (2.4/2,-1/2) -- (1/2,-2.4/2) -- (-1/2,-2.4/2) -- (-2.4/2,-1/2) -- (-2.4/2,1/2) -- (-1/2,2.4/2);
\draw[dash dot] (-1/2,2.4/2) -- (1/2,-2.4/2);
\end{tikzpicture}
\caption{Cycle $C$ of size 8 with chord inducing two cycles of size 5}
\end{figure}

However, Figure 6 cannot arise in a bipartite graph as there is no two-coloring  of the vertices in the figure when the chord is present. Therefore, one of the two cycles formed by the chord dividing $C$ must have length greater than or equal to 6 yet less than the length of $C$ itself and so we're returned to the case above.

As this process always decreases the length of the cycle $C$ and has lower bound equal to $6$, then a finite iteration of this process must achieve said bound, leading to the contradiction by Lemma 2.3.
\end{proof}

\section{Efficiency of the algorithm}

Our attention now turns to the efficiency of the problem. Algorithm 1 is a naive implementation in that faster means of checking whether a graph is strongly chordal are known, in fact it can checked via the construction of a strong perfect elimination ordering. This construction is related to the problem of doubly lexical orderings of matrices and can be performed in time $O(\min(s^2,(s+t) \log s))$ where $s$ is the cardinality of the vertex set and $t$ is the cardinality of the edge set \cite{Lubiw} \cite{Tarjan}. This allows us to check that a bipartite graph is chordal bipartite in $O(\min(s^2,(2s+t) \log(s)) = O(\min(s^2,(s+t) \log s)$ time (as a graph is chordal bipartite if and only if the graph obtained by adding an edge between every pair of vertices in one of its partitions is strongly chordal). Thus, the decision problem is in complexity class P. The algorithm for determining whether a bipartite graph is 2-chordal bipartite amounts to iterating the chordal bipartite algorithm over the edge set of the bipartite graph. Within the loop the subgraph has one fewer edge but the same number of vertices.

\begin{thm}[2-chordal Bipartite Efifciency]
Given a bipartite graph $G$, it can be checked whether or not it is 2-chordal bipartite in polynomial time on the vertex set alone.
\end{thm}

\begin{proof}
Let $G$ be our graph with $s$ vertices and $t$ edges. Let $D_{s,t}$ be the time it takes to produce an induced subgraph from $G$ (can be taken to be $s+t-1$ assuming operations are done on vertices and edges) and let $T_{s,t}$ be the time it takes to check whether $G$ is chordal bipartite. To check that $G$ is 2-chordal bipartite, we first check that $G$ is itself chordal bipartite and then check that all of its subgraphs created by deleting an edge are chordal bipartite. Then determining whether $G$ is 2-chordal bipartite can be done in
\[
T_{s,t} + t(D_{s,t}+T_{s,t-1})
=
O(
\max \left\lbrace
T_{s,t},
t D_{s,t},
t T_{s,t-1}
\right\rbrace )
\]
\[
=
O(
\max \left\lbrace
\min(s^2,(s+t) \log s)
,
t D_{s,t} ,
\min(ts^2,(t(s+t-1))\log s)
\right\rbrace
)
\]
Therefore the 2-chordal bipartite problem is polynomial on the vertex and edge set. To show that it is polynomial on the vertex set, we note that for bipartite graphs $t$ is bounded by $\frac{s^2}{4}$. This is because the complete bipartite graph $K_{s_1,s_2}$ contains $s_1 s_2$ edges and for a fixed number of total vertices, i.e. $s = s_1+s_2$, the maximum is obtained when $s_1=s_2$ (or $s_1 = s_2+1$ if $s$ is odd). These maxima are $\frac{s^2}{4}$ and $\frac{s^2-1}{4}$ respectively. So by replacing every $t$ above with $s^2/4$ we obtain a polynomial bound on the runtime in terms of the cardinality of the vertex set alone.
\end{proof}

Going a bit further, it's reasonable to assume that $D_{s,t}$ is negligible compared to the other operations and that $t>1$, in which case the above is $O(\min \lbrace ts^2,t(s+t-1) \log s)$. In other words, the process is as efficient as possible with the current characterization of 2-chordal bipartite graphs offered here. Like the $2$-chordal bipartite problem, the $3$-chordal bipartite problem, and the general $k$-chordal bipartite problem, is polynomial in complexity on the vertex set alone and its runtime can be found by iterating this process giving
\[
O \left(
T_{s,t-k+1} \prod_{i=0}^{k-2} (t-i)
\right)
\]
under the mild assumptions given above. A small corollary of Theorem 5.1 is that the induced subgraph problem for the double square graph is also in P.

\bibliography{mybibfile}{}
\bibliographystyle{plain}

\end{document}

%% file: macros.tex
\usepackage{graphicx}
\usepackage{amsmath,amsthm,amsfonts,amscd,amssymb,comment,eucal,latexsym,mathrsfs}
\usepackage{stmaryrd}
\usepackage[all]{xy}

\usepackage{epsfig}

\usepackage[all]{xy}
\xyoption{poly}
\usepackage{fancyhdr}
\usepackage{wrapfig}
\usepackage{epsfig}
\usepackage{titletoc,tocloft}
\usepackage{hyperref}
\usepackage{enumitem}
\usepackage{tikz-cd}
\usetikzlibrary{cd}

\theoremstyle{plain}
\newtheorem{thm}{Theorem}[section]

\newtheorem{lem}[thm]{Lemma}

\newtheorem*{thm*}{Theorem}

\theoremstyle{definition}

\theoremstyle{remark}

\advance \topmargin by -\headheight
\advance \topmargin by -\headsep
\evensidemargin \oddsidemargin
\usepackage{tikz}

\definecolor{cof}{RGB}{219,144,71}
\definecolor{pur}{RGB}{186,146,162}
\definecolor{greeo}{RGB}{91,173,69}
\definecolor{greet}{RGB}{52,111,72}

             \def\N{{\mathbb{N}}}            

  \newcount\colveccount
\newcommand*\colvec[1]{
        \global\colveccount#1
        \begin{bmatrix}
        \colvecnext
}
\def\colvecnext#1{
        #1
        \global\advance\colveccount-1
        \ifnum\colveccount>0
                \\
                \expandafter\colvecnext
        \else
                \end{bmatrix}
        \fi
}